\documentclass[11pt]{amsart}

\usepackage{amssymb,amsmath,amscd,amsthm,enumerate,verbatim}
\usepackage[utf8]{inputenc}
\usepackage[mathscr]{euscript}
\usepackage{bbm}
\usepackage[left=1.3in,right=1.3in,top=1in,bottom=1in]{geometry}

\usepackage{cite}

\date{}

\newcommand{\A}{{\mathcal A}}
\newcommand{\Hi}{{\mathcal H}}
\newcommand{\Z}{{\mathbbm Z}}
\newcommand{\R}{{\mathbbm R}}
\newcommand{\C}{{\mathbbm C}}

\newcommand{\I}{{\mathbbm I}}

\newcommand{\X}{{\mathbbm X}}

\newcommand{\Fib}{{\mathrm F}}

\newcommand{\TM}{{\mathrm{TM}}}

\newtheorem{theorem}{Theorem}[section]
\newtheorem{lemma}[theorem]{Lemma}

\theoremstyle{definition}
\newtheorem{remark}[theorem]{Remark}
\theoremstyle{definition}

\theoremstyle{definition}

\theoremstyle{definition}

\allowdisplaybreaks

\numberwithin{equation}{section}

\newcommand{\set}[1]{\left\{#1\right\}}
\newcommand{\eqdef}{\overset{\mathrm{def}}=}

\begin{document}

\title[Resolvent Methods for Quantum Walks]{Resolvent Methods for Quantum Walks with an Application to a Thue-Morse Quantum Walk}

\author[J.\ Fillman]{Jake Fillman}

\maketitle

\begin{abstract}
In this expository note, we discuss spatially inhomogeneous quantum walks in one dimension and describe a genre of mathematical methods that enables one to translate information about the time-independent eigenvalue equation for the unitary generator into dynamical estimates for the corresponding quantum walk. To illustrate the general methods, we show how to apply them to a 1D coined quantum walk whose coins are distributed according to an element of the Thue--Morse subshift.
\end{abstract}


\section{Introduction}

Quantum walks in one dimension comprise a popular class of models in mathematics, physics, and computer science; see \cite{AVWW, BGVW,CGMV, CGMV2, CedWer16, DFO, DFV, DMY2, GVWW, J11, J12, K14, KS11, KS14, SK2010, ST12} and references therein for a small sampler on the literature. A (one-dimensional coined) \emph{quantum walk} is described by a unitary operator on the Hilbert space
\[
\mathcal{H} \eqdef \ell^2(\Z) \otimes \C^2,
\]
 which models a state space in which a wave packet comes equipped with a ``spin'' at each integer site. The elementary tensors of the form
\[
\delta_n^\pm \eqdef \delta_n \otimes e_\pm,
\quad n \in \Z,
\]
comprise an orthonormal basis of $\Hi$, where we denote the standard orthonormal basis of $\C^2$ by
\[
e_+ \eqdef \begin{bmatrix}
1 \\ 0
\end{bmatrix},
\quad
e_-
\eqdef
\begin{bmatrix}
0 \\ 1
\end{bmatrix}.
\]
A time-homogeneous quantum walk is given as soon as coins
\begin{equation}\label{e.timehomocoins}
C_{n}
=
\begin{pmatrix}
c^{11}_{n} & c^{12}_{n} \\
c^{21}_{n} & c^{22}_{n}
\end{pmatrix}
\in \mathbbm U(2), \quad n \in \Z,
\end{equation}
are specified. As one passes from time $\ell$ to time $\ell + 1$, the update rule of the quantum walk is given by $U = SC$, where $S$ denotes the shift $S\delta_n^\pm = \delta^\pm_{n \pm 1}$, and
\[
C \eqdef \bigoplus_{n \in \Z} C_n.
\]
That is to say, $C$ is defined by
\[
C: \delta_{n}^+ 
\mapsto
  c^{11}_{n} \delta_{n}^+ 
+ c^{21}_{n} \delta_{n}^- , \quad
C: \delta_n^-  
\mapsto
  c^{12}_{n} \delta_{n}^+ 
+ c^{22}_{n} \delta_{n}^-  .
\]
Given an initial state $\psi \in \Hi$ normalized by $\| \psi \| = 1$, we are interested in the time evolution of the vector $\psi$, that is, we want to study the evolution of $\psi(\ell) = U^\ell \psi$ as $\ell \in \Z_+$ grows. The most favorable situations are those in which the walk is \emph{translation-invariant}, i.e., there exists $q$ such that $C_{n+q} = C_n$ for all $n$. In this case, one can explicitly solve the walk via a Floquet--Bloch transform, and one deduces strong ballistic motion and an explicit expression for the asymptotic group velocity; see \cite[Theorem~4]{AVWW} and \cite[Corollary~9.3]{DFO}. In fact, ballistic motion with an explicit group velocity also holds for quantum walks that are rapidly and uniformly approximated by translation-invariant quantum walks~\cite[Remark~1.3.(4)]{FillmanCMP}.

Naturally, one wants to go beyond exactly solvable models. In this short note, we will give a brief introduction to the general methods of~\cite{DFO}. Specifically, we will describe a general method that enables one to study the time-dependent spreading characteristics of a quantum walk with spatially inhomogeneous coins by establishing suitable estimates on the time-independent eigenvalue equation. Concretely, one must estimate matrix elements of $R_U(z) = (U-z\I)^{-1}$, the \emph{resolvent} of the unitary generator $U$. This is particularly fruitful in one dimension, where we can probe the resolvent by using a transfer matrix formalism to study solutions of the time-independent eigenvalue equation.

In this paper, we focus on the case $\psi = \delta_0^+$ for simplicity, but it is a straightforward matter to generalize this to different initial states. To quantify the spreading rate of $U^\ell \delta_0^+$, we first put
$$
a(n,\ell) 
= 
\left| \left\langle \delta_n^+ , U^\ell \delta_0^+ \right\rangle \right|^2
+ \left| \left\langle \delta_n^- , U^\ell \delta_0^+ \right\rangle \right|^2,
\quad
\ell \in \Z_+, \, n \in \Z,
$$
which can be thought of as the probability that the wavepacket is at site $n$ at time $\ell$. The astute observer will notice that both terms that comprise $a(n,\ell)$ correspond to the modulus squared of a Fourier coefficient of a spectral measure of $U$. Since regularity estimates on measures enable one to prove decay estimates on their Fourier coefficients, one can prove quantitative estimates on wavepacket propagation by investigating regularity and singularity of spectral measures of $U$, an approach that is proposed and executed in~\cite{DFV}. However, there are two nontrivial drawbacks to this method. First, in some cases, the spectral measures are so singular that the kinds of estimates that one can deduce from spectral regularity are essentially useless; this is the case in random polymer models, for which one expects pure point spectral measures. Second, it is often quite difficult in concrete examples to prove quantitative estimates on the regularity of spectral measures. This is the case for quantum walks with Thue--Morse coins discussed herein; it is known that their spectral measures are singularly continuous, but we have no quantitative information regarding the moduli of continuity.

In one-dimensional models, regularity estimates on spectral measures may be established using \emph{subordinacy theory}, which relates growth and decay estimates on generalized eigenfunctions to quantitative regularity of spectral measures. The quantitative theory of subordinacy is described for CMV matrices in~\cite[Section~10.8]{S2}, so, in light of the CGMV connection~\cite{CGMV,CGMV2}, this analysis also applies to 1D coined quantum walks. Thus, one naturally may wonder whether it is possible to ``cut out the middleman'' of subordinacy theory and to go directly from eigenfunction estimates to dynamical information. One avenue that enables one to do precisely this is the following variant of the Parseval formula.

\begin{theorem}[\mbox{Damanik, F., Vance~\cite[Proposition~3.16]{DFV}}] \label{t:dfv}
Let $\Hi$ be a Hilbert space, and $\varphi, \; \psi \in \Hi$. If $U$ is a unitary operator on $\Hi$, one has
\begin{equation} \label{eq:pars}
\sum_{\ell = 0}^\infty e^{-2\ell/L} |\langle \psi, U^\ell \varphi \rangle |^2
=
e^{2/L} \int_0^{2\pi} \left| \left\langle \psi, \left(U-e^{i\tau + L^{-1}} \I_{\Hi} \right)^{-1} \varphi \right\rangle \right|^2 \, \frac{d\tau}{2\pi}
\end{equation}
for all $L > 0$.
\end{theorem}

The identity \eqref{eq:pars} allows us to connect (time-averaged) dynamical information (on the left hand side) to generalized eigenfunctions, since the matrix elements of the resolvent may (almost) be recovered from the eigenvalue equation. For example, if $\phi = (U-z)^{-1} \delta_0^+$, then $U\phi = z\phi + \delta_0^+$, so $\phi$ is ``almost'' an eigenvector of $U$. In light of the averaging present on the left-hand side of Theorem~\ref{t:dfv}, we will be interested in the time-averaged dynamical probabilities, given by
\begin{equation}\label{e.averagedprob}
\widetilde{a}(n,L) = \left(1-e^{-2/L}\right) \sum_{\ell=0}^{\infty} e^{-2\ell/L} a(n,\ell),
\quad
L > 0, \, n \in \Z.
\end{equation}
Since $U$ is a unitary operator, $\| U^\ell \psi \| = 1$ for every $\ell$, so we may view $a(\cdot,\ell)$ and $\widetilde a(\cdot,L)$ as defining a pair of probability distributions on $\Z$. From this point of view, it makes sense to describe the spreading of these distributions in terms of their moments. In this paper, we will focus on the time-averaged setting, and so we define
\[
\left\langle |X|^p \right\rangle(L)
=
\sum_{n \in \Z}  (|n|+1)^p \, \widetilde{a}(n,L),
\quad
L, \; p > 0.
\]
To better understand the moments $|X|^p$, let us consider their behavior in two simple cases.
\begin{enumerate}
\item First, consider $U = \I_{\ell^2(\Z)}$, the identity operator. Clearly, $U^\ell \delta_0^+ = \delta_0^+$ for all $\ell$. Consequently, $\langle|X|^p\rangle(L) = 1$ for all $L$. Thus, $|X|^p$ remains bounded in this case.

\item At the other extreme, consider $U = S$, the shift defined above. Clearly, then, $U^\ell \delta_0^+ = \delta_\ell^+$ for $\ell \in \Z_+$, so $\langle|X|^p\rangle(L)$ grows like $L^p$ as $L\to \infty$. 
\end{enumerate}
 In order to compare the growth of the $p$th moment to polynomial growth of the form $L^{\beta p}$ for a suitable exponent $\beta$, we define (upper and lower) \emph{transport exponents} by
\[
\widetilde{\beta}^+(p) 
= 
\limsup_{L \to \infty} \frac{\log \left( \left\langle |X|^p \right\rangle (L) \right)}{p \log(L)}, 
\quad 
\widetilde{\beta}^-(p) 
= \liminf_{L \to \infty} \frac{\log \left( \left\langle|X|^p\right\rangle(L) \right)}{p \log(L)}.
\]
The values for $\widetilde \beta^\pm$ range from zero to one. In view of the previous examples, $\widetilde \beta = 1$ represents (time-averaged) ballistic transport and $\widetilde \beta = 0$ represents (a weak form of) dynamical localization. Notice, by Jensen's inequality,  $\widetilde{\beta}^{\pm}$ are both non-decreasing functions of $p > 0$~\cite[Lemma~2.7]{DT2010}.

\subsection{Thue--Morse Coins}

As a concrete example of how one can implement the overall program relating eigenfunction estimates to dynamical estimates, we will investigate the transport exponents for spatially inhomogeneous quantum walks with coins modulated by the Thue--Morse subshift.  We begin by describing the Thue--Morse substitution. Let $\A = \{0,1\}$. The \emph{free monoid} over $\A$ denotes the set of words of finite length with letters drawn from $\A$, and is denoted by $\A^*$. The \emph{Thue--Morse substitution} is the map $S:\A \to \A^*$ defined by $S(0) = 01$ and $S(1) = 10$. This extends by concatenation to maps $S:\A^* \to \A^*$, and $S:\A^{\Z_+} \to \A^{\Z_+}$. We may then iterate $S$ on the initial symbol 0 to obtain a sequence of words $w_n \eqdef S^n(0)$:
\begin{align*}
w_0 & = 0 \\
w_1 & = S(0) = 01 \\
w_2 & = S(01) = S(0)S(1) = 0110 \\
w_3 & = S(0110) = S(0)S(1)S(1)S(0) = 01101001 \\ 
\cdots
\end{align*}
Notice that $w_n$ is always a prefix of $w_{n+1}$, and thus, there is a natural limiting infinite word
\[
w_\TM
\eqdef
0110100110010110 \cdots
\in \A^{\Z_+},
\]
and it is easy to see that $S(w_\TM)=w_\TM$, i.e., $w_\TM$ is invariant under the action of $S$. The associated \emph{subshift} is defined to be the set of all two-sided infinite $0$-$1$ sequences that have the same local factor structure as $w_\TM$, that is,
\[
\X_{\TM}
=
\set{x:\Z \to \A : \text{ every finite subword of } x \text{ is a subword of } w_\TM}.
\]
One may also view $\X_\TM$ as the collection of accumulation points of $\{ T^n w_\TM : n \ge 1\}$, where $T$ denotes the left shift 
\[
[Tx]_n \eqdef x_{n+1},
\quad
n \in \Z, \; x \in \A^{\Z}.
\]
Aperiodic subshifts such as $\X_\TM$ are popular as tractable models of one-dimensional \emph{quasicrystals}, i.e., mathematical structures that simultaneously exhibit \emph{aperiodicity} (the absence of translation-invariance) and \emph{long-range order} \cite{BaDaGr}. For a much more thorough introduction into mathematical aspects of quasicrystals, the reader is referred to~\cite{BaakeGrimm}. The Fibonacci substitution, defined by $S_\Fib(0) = 01$ and $S_\Fib(1) = 0$ is another popular quasi-crystal model. Quantum walks with Fibonacci coins were studied numerically in \cite{RMM} and mathematically in \cite{DFO,DFV,DMY2}.

 Now, pick phases $\theta,\phi \in \R$. For each $x \in \X_\TM$, we obtain a quantum walk operator $U = U_{x,\theta,\phi}$ by choosing coins
\begin{equation} \label{eq:pdcoins}
C_n
=
C_{n,x}
=
\begin{cases}
R_\phi   & \text{if }x_n = 0 \\
R_\theta & \text{if } x_n = 1,
\end{cases}
\end{equation}
where
\[
R_\gamma
\eqdef
\begin{bmatrix}
\cos \gamma & -\sin \gamma \\
\sin \gamma &  \cos \gamma
\end{bmatrix}
\quad
\gamma \in \R.
\]

\begin{theorem}\label{t:pd}
Fix $\theta,\phi \in (0,\pi/2)$ and $x \in \X_\TM$, and denote by $U = U_{x,\theta,\phi}$ the quantum walk operator with coins $C_n$ defined by \eqref{eq:pdcoins}, and denote the corresponding transport exponents by $\widetilde \beta^\pm(p) = \widetilde \beta_{x,\theta,\phi}^\pm(p)$. For all $\theta, \;\phi$, all $x \in \X_\TM$, and all $p > 0$, we have
\[
\widetilde \beta^+(p)
\ge 
\widetilde \beta^-(p)
\ge 
1- \frac{1}{p}.
\]
\end{theorem}

We remark that Theorem~\ref{t:pd} is not new. Indeed, it is essentially~\cite[Theorem~8.2]{DFO}. However,~\cite{DFO} produces this result as a corollary of a fairly general piece of machinery,~\cite[Theorem~2.1]{DFO}. Thus, the goal of the present note is to illustrate the general manner in which resolvent estimates may be transformed into dynamical estimates in a simple (but nontrivial) scenario, so we will show how to prove the statement directly, without appealing to the general machinery of~\cite{DFO}.

\section{Proof of Lower Bounds}

In this section, we will describe the arguments that prove Theorem~\ref{t:pd}. There are two key ingredients: an exact renormalization scheme for eigenfunctions of $U$ and the presence of degenerate spectral parameters at which the corresponding transfer operators commute with one another. The renormalization scheme is most clearly exhibited by the so-called \emph{transfer matrices}, which we now introduce. Suppose that 
\begin{equation} \label{eq:eigeq}
U \psi = z \psi \text{ for some } z \in \C \text{ and some } \psi \in \C^\Z \otimes \C^2.
\end{equation}
We may represent $\psi$ in coordinates as
\[
\psi
=
\sum_{n \in \Z} \left( \psi_n^+  \delta_n^+ + \psi_n^- \delta_n^- \right),
\quad
\psi_n^\pm \in \C.
\]
Notice, we do \emph{not} assume that $\psi$ is a normalizable state in $\Hi$, i.e., we do \emph{not} assume
\[
\sum_{n \in \Z} \left( |\psi_n^+|^2 + |\psi_n^-|^2 \right) 
<
\infty.
\]
It is straightforward to deduce from \eqref{eq:eigeq} that
\begin{equation} \label{eq:eigrec}
\begin{bmatrix}
\psi_{n+1}^+ \\
\psi_n^-
\end{bmatrix}
=
M_x(n;z)
\begin{bmatrix}
\psi_n^+ \\
\psi_{n-1}^-
\end{bmatrix}
\text{ for all } n \in \Z
\end{equation}
for fixed $x \in \X_\TM$, where $M$ is defined by
\begin{equation} \label{eq:tmdef}
M_x(n;z)
\eqdef
\sec(\varphi_n)
\begin{bmatrix}
z^{-1} & -\sin(\varphi_n) \\
-\sin(\varphi_n) & z
\end{bmatrix},
\quad
\varphi_n
\eqdef
\begin{cases}
\phi & \text{ if } x_n = 0 \\
\theta & \text{ if } x_n = 1
\end{cases}
\end{equation}
Thus, we are concerned products involving the matrices:
\[
A_\phi(z)
\eqdef
\sec(\phi)
\begin{bmatrix}
z^{-1} & -\sin(\phi) \\
-\sin(\phi) & z
\end{bmatrix}
\quad
A_\theta(z)
\eqdef
\sec(\theta)
\begin{bmatrix}
z^{-1} & -\sin(\theta) \\
-\sin(\theta) & z
\end{bmatrix}
\]
In particular, to control growth and decay of eigenfunctions of $U$, it suffices to control the growth and decay of the norms of matrices of the form
\[
T_x(n,m;z)
=
\begin{cases}
\I_{2 \times 2} & n = m \\
M_x(n-1;z) M_x(n-2;z) \cdots M_x(m;z) & n > m \\
T_x(m,n;z)^{-1} & n < m
\end{cases}
\]

Thus, the key estimate is supplied by the following preliminary result.

\begin{lemma} \label{l:tmbounds}
For each $0 < \theta < \pi/2$, there is a constant $c = c(\theta,\phi) > 0$ such that
\begin{equation} \label{pdtmbounds}
\| T_x(n,m;i) \|
\leq
c
\end{equation}
for all $n,m \in \Z$ and all $x \in \X_\TM$.
\end{lemma}

\begin{proof}
It is straightforward to check that 
\[
A_\theta(i) A_\phi(i)A_\phi(i) A_\theta(i) 
=
A_\phi(i)A_\theta(i)  A_\theta(i) A_\phi(i)
=
\I_{2\times 2}.
\]
Since any $x \in \X_\TM$ may be uniquely decomposed into subwords of the form ``0110'' and ``1001'', the lemma follows immediately by interpolation.
\end{proof}

\begin{remark}
In the language of \cite{LQY2016}, Lemma~\ref{l:tmbounds} shows that $z=i$ is a ``Type-I'' spectral parameter for the Thue--Morse quantum walk. We chose $z=i$ for concreteness, but there is a countably dense subset of the spectrum upon which the transfer matrices are bounded in this fashion; compare \cite[Lemma~8.1]{DFO}.
\end{remark}

\begin{proof}[Proof of Theorem~\ref{t:pd}]

The theorem follows is a consequence of Lemma~\ref{l:tmbounds} and \cite[Theorem~2.1]{DFO}. Since this paper is meant to be an illustration of the methods, though, let us say a few more words about how the arguments develop. Throughout, we use $f \gtrsim g$ to denote that there is a constant $c$ such that $f \geq c g$; constants may depend on $\theta$, $\phi$, and $p$, but will be independent of $n$, $L$, $x$, and $\varepsilon$.

Notice that the right-hand side of the Parseval identity~\eqref{eq:pars} involves an integral over $\tau$, so we need effective estimates not just at $z=i$, but on a set of phases with positive length. Thus, the first step is to perturb Lemma~\ref{l:tmbounds} in the quasi-energy; concretely, Lemma~\ref{l:tmbounds} gives information when the spectral parameter is $z = i$, and we can parlay this into information about spectral parameters $z$ that are very close to $z=i$. We may do this at the expense of restricting the range of $n$ and $m$ to suitable finite-length windows; such ideas have a long and venerable history in differential equations and usually go by the name \emph{Gronwall's inequality}. The discrete version that we need is furnished by~\cite{DFO}. In particular, by Lemma~\ref{l:tmbounds} and~\cite[Lemma~3.3]{DFO},\footnote{The matrices in~\cite[Lemma~3.3]{DFO} are the so-called Szeg\H{o} matrices, which are not exactly the same as those introduced in \eqref{eq:tmdef}. However, these matrices are conjugate to one another in a natural manner; cf.\ \cite[Equation~(3.3)]{DFO}.} there is a constant $K_1$ (independent of $n,m,z$, and $\varepsilon$), so that
\begin{equation} \label{eq:pert:tmest}
\| T_x(n,m;z) \|
\leq K_1
\end{equation}
whenever $|z-i| < \varepsilon$ and $|n-m| < \varepsilon^{-1}$ for some $0 < \varepsilon < 1$. Since $T_x(n,m;z)$ is unimodular, we get
\begin{equation} \label{eq:pert:vecest}
\|T_x(n,m;z) \vec v \|
\ge
K_1^{-1}
\text{ whenever }
|z-i| < \varepsilon, \; |n-m| \le \varepsilon^{-1}, \text{ and } \|\vec v\| = 1.
\end{equation}
For each $z$ with $|z| > 1$, define $\psi = \psi_z \eqdef (U-z)^{-1}\delta_0^+$. Since $U\psi = z\psi + \delta_0^+$, one may check that the recursion~\eqref{eq:eigrec} holds for all $n \neq -1$.  In light of this and \eqref{eq:pert:vecest}, we may estimate the matrix elements of the resolvent via
\begin{equation} \label{eq:resolventest}
\left|\langle \delta_n, (U-e^{i\tau + \eta})^{-1} \delta_0^+ \rangle\right|^2
\gtrsim
1
\; \text{ whenever } |e^{i\tau + \eta} - i| < \varepsilon \text{ and }
\frac{\varepsilon^{-1}}{2} < |n| < \varepsilon^{-1},
\end{equation}
where $\tau \in \R$ and $\eta > 0$. In~\eqref{eq:resolventest} and below, we use the shorthand notation
\[
|\langle \delta_n, \varphi\rangle|^2
\eqdef
|\langle \delta_n^+, \varphi\rangle|^2 + |\langle \delta_n^-, \varphi\rangle|^2.
\]
Now, we combine everything to estimate the time-averaged moments. First, use the definition and remove the terms with $|n| < 2L$ and $|n| > 4L$ to obtain
\[
\left\langle |X|^p\right\rangle(L)
=
\sum_{n \in \Z} |n|^p \, \widetilde a(n,L)
\geq 
L^p \left(1-e^{-2/L}\right) \sum_{2L \le |n| \le  4L} \sum_{\ell = 0}^\infty e^{-2\ell/L} \left|\langle \delta_n, U^\ell \delta_0^+ \rangle\right|^2.  
\]
Of course, $1-e^{-2/L} \sim L^{-1}$ as $L \to \infty$, so we may estimate this as
\[
L^p \left(1-e^{-2/L}\right) \sum_{2L \le |n| \le  4L} \sum_{\ell = 0}^\infty e^{-2\ell/L} \left|\langle \delta_n, U^\ell \delta_0^+ \rangle\right|^2  
\gtrsim
L^{p-1} \sum_{2L \le |n| \le  4L} \sum_{\ell = 0}^\infty e^{-2\ell/L} \left|\langle \delta_n, U^\ell \delta_0^+ \rangle\right|^2  .
\]
Now, apply the Parseval formula~\eqref{eq:pars} to get
\[
L^{p-1} \sum_{2L \le |n| \le  4L} \sum_{\ell = 0}^\infty e^{-2\ell/L} \left|\langle \delta_n, U^\ell \delta_0^+ \rangle\right|^2 
\gtrsim
L^{p-1} \sum_{2L \le |n| \le  4L} \int_0^{2\pi} \left|\langle \delta_n, (U-e^{i\tau + L^{-1}})^{-1} \delta_0^+ \rangle\right|^2 \, \frac{d\tau}{2\pi}.
\]
We want to concentrate on phases $\tau$ for which $e^{i \tau} \approx i$, so we can bound this from below by
\[
L^{p-1} \sum_{2L \le |n| \le  4L} \int_{B_L} \left|\langle \delta_n, (U-e^{i\tau + L^{-1}})^{-1} \delta_0^+ \rangle\right|^2 \, \frac{d\tau}{2\pi} ,
\]
where $B_L = [\pi/2-L^{-1},\pi/2 + L^{-1}]$. At last, applying \eqref{eq:resolventest}, and using that the length of $B_L$ is $2/L$, we have
\begin{align*}
L^{p-1} \sum_{2L \le |n| \le  4L} \int_{B_L} \left|\langle \delta_n, (U-e^{i\tau + L^{-1}})^{-1} \delta_0^+ \rangle\right|^2 \, \frac{d\tau}{2\pi}
\gtrsim
L^{p-1} \cdot L \cdot L^{-1} 
=
L^{p-1}.
\end{align*}
Thus, we have
\[
\left\langle |X|^p\right\rangle(L)
\gtrsim
L^{p-1},
\]
which suffices to prove the desired lower bound on $\widetilde \beta^-(p)$.

\end{proof}

\section*{Acknowledgements}

I am grateful to the Workshop on Quantum Simulation and Quantum Walks for the invitation to participate and for their generous hospitality during my visit in Yokohama. This work was supported in part by an AMS--Simons Travel Grant, 2016--2018.

\end{document}